\documentclass[12pt,english,reqno]{amsart}
\usepackage{amscd,amssymb,amsmath,amstext,amsthm,amsfonts,bbold}
\usepackage[dvips]{graphicx}
\usepackage{mathrsfs}
\usepackage{float}

\usepackage{stmaryrd}
\usepackage{a4wide}
\usepackage[pdftex]{hyperref}
\usepackage[usenames]{color}
\usepackage[ansinew]{inputenc}

\newtheorem{Theorem}{Theorem}

\newtheorem{maintheorem}{Theorem}

\newtheorem{maincorollary}[maintheorem]{Corollary}

\newtheorem{T}{Theorem}
\newtheorem{Corollary}[T]{Corollary}
\newtheorem{Proposition}[T]{Proposition}
\newtheorem{Lemma}[T]{Lemma}

\newtheorem{Remark}[T]{Remark}
\newtheorem*{remark}{Remark}
\newtheorem{Definition}[T]{Definition}

\newtheorem*{claim}{Claim}

\newtheorem{Claim}{Claim}

\def \AA {{\mathbb A}}
\def \BB {{\mathbb B}}

\def \RR {{\mathbb R}}

\def \NN {{\mathbb N}}

\def \VV {{\mathbb V}}

\def \XX {{\mathbb X}}

\def \RE {{\mathrm{Re}}}

\def \cf {\mathcal{F}}

\def \cp {\mathcal{P}}
\def \cc {\mathcal{C}}

\def \cu {\mathcal{U}}
\def \co {\mathcal{O}}

\def \cv {\mathcal{V}}

\def \ci {\mathcal{I}}

\def \cR {\mathscr{R}}

\newcommand{\leb}{\operatorname{Leb}}
\newcommand{\dist}{\operatorname{dist}}

\newcommand{\per}{\operatorname{Per}}

\newcommand{\interior}{\operatorname{Interior}}
\newcommand{\fix}{\operatorname{Fix}}

\begin{document}

\author{P. Brand\~ao}
\address{P. Brand\~ao\\
Impa, Estrada Dona Castorina, 110, Rio de Janeiro, Brazil.} \email{paulo@impa.br}

\author{J. Palis}
\address{J. Palis\\
Impa, Estrada Dona Castorina, 110, Rio de Janeiro, Brazil.} \email{jpalis@impa.br}

\author{V. Pinheiro}
\address{V. Pinheiro\\
Departamento de Matem\'atica, Universidade Federal da Bahia\\
Av. Ademar de Barros s/n, 40170-110 Salvador, Brazil.}
\email{viltonj@ufba.br}

\date{\today}


\title{On the Finiteness of Attractors for piecewise $C^2$ Maps of the Interval}

\maketitle
\begin{abstract}

We consider piecewise $C^2$ non-flat maps of the interval and show that, for Lebesgue almost every point, its omega-limit set is either a periodic orbit, a cycle of intervals or the closure of the orbits of a subset of the critical points. In particular, every piecewise $C^2$ non-flat map of the interval displays only a finite number of non-periodic attractors.

\end{abstract}

\tableofcontents

\section{Introduction}

When studying dynamical systems, we often focus on attractors since many orbits converge to them in the future. Thus, it is of particular importance to obtain results on their finiteness and, as conjectured in \cite{Pa00,Pa2005}, that should be the case for a dense subset of differentiable dynamics on compact manifolds.

For maps of the interval, the finiteness of the number of attractors began to be established by Blokh and Lyubich in \cite{BL89eg, BL91}. They proved that a  $C^3$ non-flat map
with a single critical point and negative Schwarzian derivative has a single attractor, whose basin of attraction contains Lebesgue almost every point of the interval.

Later on, van Strien and Vargas  \cite{vanStrien:2004p221} proved that every $C^3$ non-flat map $f$ with a finite number of critical points has a finite number of non-periodic attractors.
More recently, we have obtained in \cite{Brandao:2013wf} the finiteness of attractors  for maps of the interval displaying discontinuities. For that, we have  assumed the maps to be piecewise $C^3$ with negative Schwarzian derivative. 

In the present paper we go further, proving that the finiteness of non-periodic attractors is also valid for non-flat piecewise $C^2$ maps.

Maps with discontinuities naturally arise from differentiable vector fields: non-flat piecewise $C^r$ maps, $r\ge 2$, can be obtained as the quotient by stable manifolds of Poincar\'e maps of $C^r$ dissipative flows \cite{GW}.

As a by-product of the techniques presented here, we also provide a new proof of Ma\~né's notable theorem concerning expanding sets of the interval \cite{Man85}.

\subsection{Statement of the main results}\label{SubsectionStaofMTh}

A compact set $A$ is called an {\em attractor} for a map $f$ if its basin of attraction $\beta_{f}(A):=\{x\,;\,\omega_f(x)\subset A\}$ has positive Lebesgue measure and there is no strictly smaller closed set $A' \subset A$ so that $\beta_{f}(A')$ is the same as $\beta_{f}(A)$ up to a zero measure set. Here, $\omega_f(x)$ is the positive limit set of $x$, that is, the set of accumulating points of the forward orbit of $x$.


A map $f:[0,1]\to \RR$ is called {\em non-flat} at the point $c\in[0,1]$ if there exist $\varepsilon>0$, constants $\alpha, \beta\ge1$  and $C^2$  diffeomorphisms $\phi_{0}:[c-\varepsilon,c]\to \text{Im}(\phi_{0})$ and $\phi_{1}:[c,c+\varepsilon]\to\text{Im}(\phi_{1})$ such that $$f(x)=\begin{cases}a+\big(\phi_{0}(x-c)\big)^\alpha&\text{ if }x\in(c-\varepsilon,c)\cap(0,1)\\
b+\big(\phi_{1}(x-c)\big)^\beta&\text{ if }x\in(c,c+\varepsilon)\cap(0,1)\end{cases}$$
where $a=\lim_{0<\varepsilon\to 0}f(c-\varepsilon)$ and $b=\lim_{0<\varepsilon\to 0}f(c+\varepsilon)$.

\begin{Definition}
Let $C_f\subset(0,1)$ be a finite set. A map $f:[0,1]\setminus\cc_f\to\RR$ is called {\em non-flat} if it is non-flat at every point $x\in[0,1]$. In particular, if $f$ is a diffeomorphism on a neighborhood of a point $p$, then it is non-flat at $p$.
\end{Definition}

In this work we are dealing with non-flat piecewise $C^2$ maps of the interval into itself. More precisely, non-flat maps $f:[0,1] \to [0,1]$ that are $C^2$ local diffeomorphisms, except for a finite set of non-flat points $\cc_f\subset (0,1)$. This {\em exceptional set} contains all critical points of $f$, as well as all discontinuities. 

Due to the existence of an exceptional set, we can have lateral periodic points that are not periodic. Indeed, a point $p\in[0,1]$ is called {\em right-periodic} with period $n$ for $f$ if $n$ is the smallest integer $\ell\ge1$ such that if $f^\ell(x)\searrow p$ when $x\searrow p$.
Similarly, $p$ is called {\em left-periodic} with period $n$ if $n$ is the smallest integer $\ell\ge1$ such that $f^\ell(x)\nearrow p$ when $x\nearrow p$.
We say that a point $p$ is {\em periodic-like} if it is left or right-periodic. In particular, a {\em fixed-like point} is periodic-like with period equal to one.
 On the left picture of Figure~\ref{PeriodicLikeAttractor2.png}, we have that $A=\{c\}$ is an attracting  fixed-like point with $\beta_f(A)=(0,1)$.
\begin{figure}
\begin{center}\includegraphics[scale=.25]{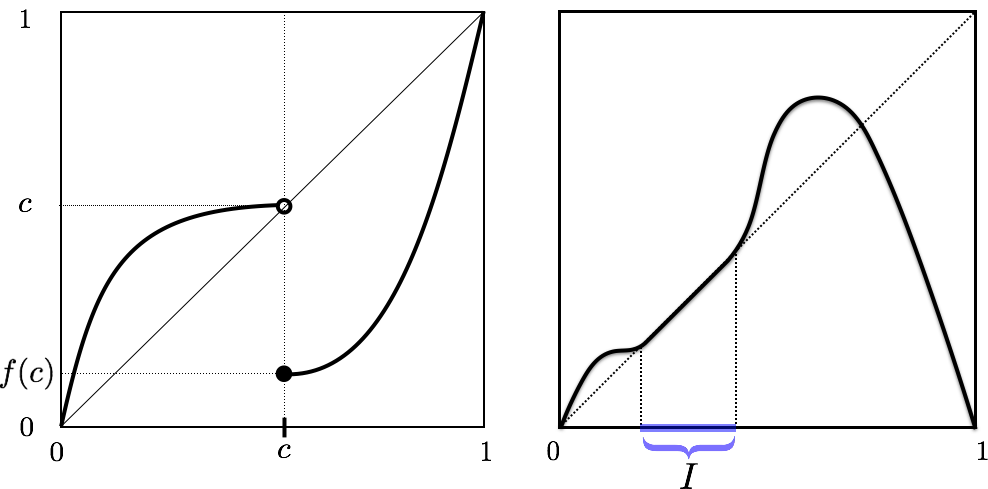}\\
\caption{\small On the right side, we see an example of an attracting periodic-like orbit (indeed, a fixed-like point) which is not an attracting periodic orbit. On the left side $I$ is an interval of fixed points.}\label{PeriodicLikeAttractor2.png}
\end{center}
\end{figure}

%
%

\begin{maintheorem}[Finiteness of the number of non-periodic attractors]\label{mainTheoremMTheoFofA}
Let $f:[0,1]\to[0,1]$, be a non-flat $C^2$ local diffeomorphism in the whole interval, except for a finite set $\cc_f\subset(0,1)$. Then, $f$ admits only a finite number of non periodic-like attractors. Indeed, there is a finite collection of attractors $A_1,\cdots,A_n$, such that
$$\leb\big(\beta_f(A_1)\cup\cdots\cup\beta_f(A_n)\cup\BB_0(f)\cup\co_f^{-}(Per(f))\big)=1,	
$$
where $\BB_0(f)$ is the union of the basin of attraction of all attracting periodic-like orbits. 
Furthermore, $\omega_f(x)=A_j$ for almost every $x\in\beta_f(A_j)$ and every $j=1,...,n$.
\end{maintheorem}

Notice that only in degenerate cases $\leb(\co_f^-(\per(f))>0$ (as in the picture on the right side of Figure~\ref{PeriodicLikeAttractor2.png}). Indeed, $\co_f^-(\per(f))$ is generically a countable set, so generically we can write the equation above as $$\leb\big(\beta_f(A_1)\cup\cdots\cup\beta_f(A_n)\cup\BB_0(f)\big)=1.$$

%
%

The theorem above combined with the fact that for every $C^2$ non-flat map of the interval there  exists $n_0\ge 1$ such that every periodic orbit of period greater or equal to $n_0$ is a hyperbolic repeller \cite{MS89}, yields the following result:

\begin{maintheorem}
[]\label{cordemelovstr}
Let $f$ be a $C^{2}$ non-flat map such that $\#(Fix f^n)<+\infty$ for every $n>0$, then $f$ has only a finite number of attractors. 
\end{maintheorem}

To state the next result, we define a cycle of intervals as a transitive finite union of non-trivial closed intervals. This is a common type of attractor for one-dimensional maps. Indeed, the support of any ergodic absolutely continuous invariant measure is always a cycle of intervals.

\begin{maintheorem}\label{mainTheoremClassification}
Let $f:[0,1]\to[0,1]$ be a non-flat $C^2$ local diffeomorphism in the whole interval, except for a finite set $\cc_f\subset(0,1)$, and let $\cv_f=\{f(c_\pm)\,;\,c\in\cc_f\}$. If $A_j$ is one of the attractors given by Theorem~\ref{mainTheoremMTheoFofA} then
$A_j$ is either a cycle of intervals
or a Cantor set of the form $A_j=\bigcup_{v\in\VV_j}\omega_f(v)$ for some $\VV_j\subset\cv_f\setminus\co_f^-(\cc_f)$ with $v\in\omega_f(v)$ $\forall\,v\in\VV_j$.
\end{maintheorem}

To prove the theorems above, we need to relate most of  the orbits with the critical set. More precisely, orbits that are not attracted to periodic points nor to cycles of intervals accumulate on the critical points. Furthermore, we also  need to assure some expansion for orbits avoiding critical points. We obtain such informations from Theorem~\ref{MainTheoPreMane}.

\begin{maintheorem}\label{MainTheoPreMane}
Let $f:[0,1]\to[0,1]$ be a non-flat $C^2$ local diffeomorphism in the whole interval, except for a finite set $\cc_f\subset(0,1)$. If $x\in[0,1]\setminus\big(\BB_0(f)\cup\co_f^-(\per(f))\big)$ and $\overline{\co_f^+(x)}\cap\cc_f=\emptyset$ then $\sup_{n}|Df^n(x)|=\infty$.
	Furthermore, $\omega_f(x)\cap\cc_f\ne\emptyset$ for Lebesgue almost every $x\in[0,1]\setminus\big(\BB_0(f)\cup\co_f^-(\per(f))\big)$.
\end{maintheorem}

For one-dimensional dynamics, Mañé exhibited in \cite{Man85} a surprisingly simple sufficient condition for orbits to have an expanding behavior: orbits that do not accumulate on periodic attractors, weak repellers nor the critical region, are expanding.  Here, as a corollary of Theorem~\ref{MainTheoPreMane}, we provide new proof of Mañé's result.

\begin{maincorollary}[Mañé]\label{ThMane}
Let $f:[a,b]\to\RR$ be a $C^{2}$ map. If $U$ is an open neighborhood of the critical points of $f$ and every periodic point of $[a,b]\setminus U$ is hyperbolic and expanding then $\Lambda=\{x\,;\,\co_f^+(x)\cap U=\emptyset\}$ is a uniformly expanding set. Furthermore, $\exists\,C>0$ and $\lambda>1$ such that, for any $n>0$, $$|Df^{n}(x)|>C \lambda^{n}$$ for every $x$ such that $\{x,\cdots,f^{n-1}(x)\}\cap U=\emptyset$. 
\end{maincorollary}

\section{Setting and preliminary facts}
\label{Sectionsettingsandpre}

Given a set $U\subset[0,1]$, the pre-orbit of $U$ is $\co_f^-(U):=\bigcup_{j\ge0}f^{-j}(U)$. If $x\notin\co_f^-(\cc_f)$, the forward orbit of $x$ is $\co_f^+(x)=\{f^j(x)\,;\,j\ge0\}$ and the $\omega$-limit set of $x$, denoted by $\omega_f(x)$, is the set of accumulating points of the sequence $\{f^n(x)\}_n$. That is, $$\omega_f(x)=\{y\in[0,1]\,;\,y=\lim_{j\to\infty}f^{n_j}(x)\text{ for some }n_j\to\infty\}.$$

If $x\in\co_f^-(\cc_f)$ then $\co_f^+(x)=\{x,\cdots,f^n(x)\}$ where $n=\min\{j\ge0\,;\,x\in f^{-j}(\cc_f)\}$.
A point $p\in[0,1]$ is called {\em wandering} if there is a neighborhood $V$ of it such that $\co_f^+(V)\cap V=\emptyset$, where $\co_f^+(V)=\bigcup_{x\in V}\co_f^+(x)$. If this is not the case, the point $p$ is called {\em non-wandering}. The {\em non-wandering set} of $f$, $\Omega(f)$, is the set of all non-wandering points $x\in[0,1]$. One can easily prove that $\Omega(f)$ is compact and that $\omega_f(x)\subset\Omega(f)$ $\forall\,x$. As usual, $L_+(f)$ denotes the closure of the union of all omega-limit sets, that is, $L_+(f)=\overline{\bigcup_{x}\omega_f(x)}$.

We denote the set of periodic points of $f$ by $\per(f)$, that is, $$\per(f)=\{p\in[0,1]\setminus\co_f^-(\cc_f)\,;\,f^n(p)=p\text{ for some }n\ge1\}.$$

As $\leb(\co_f^-(\cc_f))=0$, if $c\in\omega_f(x)$ for a Lebesgue typical point $x$, with $c\in\cc_f$, then $\co_f^+(x)\cap(c-\varepsilon,c+\varepsilon)=\co_f^+(x)\cap\,\big((c-\varepsilon,c+\varepsilon)\setminus\{c\}\big)$ $\forall\,\varepsilon>0$ and so, $\lim_{0<\varepsilon\to0}f(c+\varepsilon)$ or $\lim_{0<\varepsilon\to0}f(c-\varepsilon)$ belongs to $\omega_f(x)$ but not necessarily $f(c)$. That is, if $c\in\omega_f(x)$ the omega-limit set of such a typical point $x$ does not involve the image of the exceptional set $f(\cc_f)=\{f(c)\,;\,c\in\cc_f\}$, but their {\em lateral exceptional values}
$\cv_f=\{\lim_{\varepsilon\searrow0}f(c\pm\varepsilon)\,;\,c\in\cc_f\}.$
Because of that, we can consider $f$ as a map from $[0,1]\setminus\cc_f$ to $[0,1]$, instead of a map of the interval to itself. As a consequence, we have:
\begin{remark} To prove the theorems above, we can consider $f$ to be a $C^2$ non-flat local diffeomorphism $f:[0,1]\setminus\cc_f\to[0,1]$, where $\cc_f\subset(0,1)$ is the  exceptional set of $f$.
\end{remark}

A {\em homterval} is an open interval $J=(a,b)$ such that $f^n|_J$ is a homeomorphism for every $n\ge1$. A homterval $J$ is called a {\em wandering interval} if $J\cap\BB_0(f)=\emptyset$ and $f^j(J)\cap f^k(J)=\emptyset$ for all $1\le j<k$, where $\BB_0(f)$ is the union of the basins of attraction of all attracting periodic-like orbits of the map $f$.

\begin{Lemma}[Homterval Lemma]\label{LemmaHomterval}
Let $U$ be an open subset of $[0,1]$, $f:U\to[0,1]$ a continuous map and $I=(a,b)$ a homterval of $f$. If $I$ is not a wandering interval then $I\subset\BB_0(f)\cup\co_f^-(\per(f))$.
\end{Lemma}

\begin{Proposition}[Vargas - van Strien, \cite{vanStrien:2004p221}]\label{KoebeVvS}
Let $\cc_f\subset(0,1)$ be a finite set. If $f:[0,1]\setminus\cc_f\to[0,1]$ is a $C^2$ non-flat local diffeomorphism, then there exists a function $O(\varepsilon)$ with $O(\varepsilon)\to0$ as $\varepsilon\searrow0$ with the following property: Let $J\subset T$ be an interval, $R,L$ the connected components of $T\setminus J$ and $\delta:=\min\{|R|/|J|,|L|/|J|\}$. 
Let $n$ be an integer and $T_0 \supset J_0$ intervals such that $f^n|_{T_0}$ is a diffeomorphism, $f^n(T_0)=T$ and $f^n(J_0)=J$.  If $\varepsilon:=\max\{|f^j(T_0)|\,;\,0\le j\le n\}$, then
\begin{enumerate}
\item \label{EqVvS1}$
\big|\frac{Df^n(x)}{Df^n(y)}\big|\le\big(\frac{1+\delta}{\delta}\big)^2 e^{O(\varepsilon)\sum_{i=0}^{n-1}|f^i(J_0)|}$, for all $x,y\in J_0$;
\item \label{EqVvS2}$\exists\delta'>0$ depending only on $\varepsilon$ and $\sum_{i=0}^{n-1}|f^i(J_0)|$ such that,  for all $x,y\in J_0$ and $1\le j\le n$, we have $
\big|\frac{Df^j(x)}{Df^j(y)}\big|\le\big(\frac{1+\delta'}{\delta'}\big)^2 e^{O(\varepsilon)\sum_{i=0}^{n-1}|f^i(J_0)|}$;
\item \label{EqVvS3} $\frac{|Df^n(x)|}{|Df^n(y)|}\le\exp\big(\frac{2}{\delta|J|}|f^n(x)-f^n(y)|+O(\varepsilon)\sum_{i=0}^{n-1}|f^i(x)-f^j(y)|\big)$ for every $x,y\in J_0$.
\end{enumerate}

\end{Proposition}
Items~\ref{EqVvS1}~and~\ref{EqVvS2} of Proposition~\ref{KoebeVvS} comes directly from Proposition~2 of \cite{vanStrien:2004p221}, as we are assuming that $f^n|_{T_0}$ is a diffeomorphism. Given $x,y\in J_0$, let $J_0'=\{t x+(1-t)y\,;\,t\in[0,1]\}$, $R'$ and $L'$ be the connected components of $T\setminus f^n(J_0')$, and $$\delta':=\min\bigg\{\frac{|R'|}{|f^n(J')|},\frac{|L'|}{|f^n(J')|}\bigg\}=\frac{|J|}{|f^n(J')|}\min\bigg\{\frac{|R'|}{|J|},\frac{|L'|}{|J|}\bigg\}\ge\frac{|J|}{|f^n(J')|}\delta.$$
Applying (\ref{EqVvS1}) to $T_0,J_0'$, we get the item~\ref{EqVvS3} by noting that $$\bigg(\frac{1+\delta'}{\delta'}\bigg)^2\le\bigg(1+\frac{|f^n(J')|}{\delta|J|}\bigg)^2=\bigg(1+\frac{|f^n(x)-f^n(y)|}{\delta|J|}\bigg)^2\le\big(e^{\frac{1}{\delta|J|}|f^n(x)-f^n(y)|}\big)^2.$$

\begin{Lemma}[See \cite{Brandao:2013wf}]\label{pagina1}Let $a < b \in \mathbb{R}$ and $V \subset (a,b)$ be an open set. Let $\mathcal{P}$ be the set of connected components of a Borel set $V$. Let $G:V\to(a,b)$ be a map satisfying:
\begin{enumerate}
\item 
$G(P)=(a,b)$ diffeomorphically, for any $P \in \mathcal{P}$;
\item $\exists V'\subset \bigcap_{j\ge 0}G^{-j}(V)$, with $\leb(V')>0$, such that
\begin{enumerate}
\item $\lim_{n\to\infty}|\mathcal{P}_n(x)|=0$, $\forall x \in V'$,\\
where $\mathcal{P}_n(x)$ is the connected component of $\bigcap^n_{j=0}G^{-j}(V)$ that contains $x$;
\item\label{BD99}$\exists K>0$ such that 
$ \big|\frac{DG^n(p)}{DG^n(q)}\big| \le K,$
 for all $n$, and $p, q \in \mathcal{P}_n(x)$, and $x \in V'$
\end{enumerate}
\end{enumerate}
Then, 
$\leb([a,b]\setminus V)=0$,  $\omega_G(x)=[a,b]$ for Lebesgue almost all $x \in [a,b]$.
\end{Lemma}

\section{Wandering intervals}
The main result in this section is Proposition~\ref{PropWanderingFree}, which assures that the $\omega$-limit set of a wandering interval is always contained in the closure of the orbits of the critical values, $\overline{\co_f^+(\cv_f)}$. We also study here the existence of nice intervals outside $\overline{\co_f^+(\cv_f)}$. An interval $(a,b)$ is called {\em nice}, with respect to a local diffeomorphism $f:[0,1]\setminus\cc_f\to[0,1]$, if $(a,b)\cap\big(\co_f^+(f(a_\pm))\cup\co_f^+(f(b_\pm))\big)=\emptyset$.

\begin{Lemma}\label{LemmaInt876f}
Let $f:[0,1]\setminus\cc_f\to[0,1]$ a local diffeomorphism. If $p\in[0,1]\setminus\overline{\co_f^+(\cv_f)}$, where $\cc_f\subset[0,1]$, then one and only one of the following statements is true.
\begin{enumerate}
\item $(a,p)$ is a wandering interval for some $a<p$.
\item $p$ is a periodic-like attractor and $\beta(\co_f^+(p))\supset(a,p)$ for some $a<p$.
\item For every $\varepsilon>0$ there is $a\in(p-\varepsilon,p)$ such that $p\notin\overline{\co_f^+(a)}$.
\end{enumerate}
Similarly, either $(p,b)$ is a wandering interval for some $p<b$ or  $p$ is a periodic-like attractor and $\beta(\co_f^+(p))\supset(p,b)$ for some $p<b$ or else for every $\varepsilon>0$ there is $b\in(p,p+\varepsilon)$ such that $p\notin\overline{\co_f^+(b)}$.
\end{Lemma}
\begin{proof}
Let $r_0=\inf\{r\ge0$ $;$ $p\in\overline{\co_f^+(a)}$, $\forall\,a\in(p-r,p)\}$.
If $r_0=0$, then item (3) of the lemma is true.
Thus, suppose that $r_0>0$ and let $I=(p-r_0,p)$. As $p\notin\overline{\co_f^+(\cv_f)}$ and $p\in\omega_f(I)$, we have that $I\cap\co_f^-(\cc_f)=\emptyset$.
As a consequence, $I$ is a homterval.
Therefore, it follows from the homterval lemma that either item (1) is true or $I\subset\BB_0(f)\cup\co_f^-(\per(f))$. But, as $p\in\omega_f(I)$, $p$ must be an attracting periodic-like orbit containing $I$ in its basin of attraction.  
\end{proof}

\begin{Proposition}\label{PropWanderingFree}
Let $f:[0,1]\setminus\cc_f\to[0,1]$ be a $C^2$ non-flat local diffeomorphism, where $\cc_f\subset[0,1]$ is a finite set. If $I$ is a wandering interval then $\omega_f(I)\subset\overline{\co_f^+(\cv_f)}$.
\end{Proposition}
\begin{proof}
Let $J$ be the maximal open wandering interval containing $I$ and suppose that $p\in\omega_f(J)$ for some $p\in[0,1]\setminus\overline{\co_f^+(\cv_f)}$. 
Let $T$ be the connected component of $[0,1]\setminus\overline{\co_f^+(\cv_f)}$ containing $p$.
We may assume that $p\in\overline{\co_f^+(I)\cap(0,p)}$, the case $p\in\overline{\co_f^+(I)\cap(p,1)}$ is analogue.

We will consider two cases depending on $p$ belonging or not to the boundary of a wandering interval.
First consider the case in which $p$ does not belong to a boundary of a wandering interval.

In this case, as $p\in\overline{\co_f^+(I)\cap(0,p)}$, given any $\forall\,0\le a<x$, neither $(a,x)$ is a wandering interval  nor $p$ is an attracting periodic orbit with $\beta(\co_f^+(p))\supset(a,p)$.
 It follows from Lemma~\ref{LemmaInt876f} that there are sequences $a_n\nearrow p$ and $p\le b_n$ such that $\co_f^+(a_n)\cap(a_n,b_n)=\emptyset$.
 If $p\notin\omega_f(p)$ or if $p\in\per(f)$, set $(\alpha_n,\beta_n)=(a_n,p)$.

Suppose that $p$ is recurrent, $p\in\omega_f(p)$, but not a periodic point. As we are assuming that $p$ does not belong to the boundary of a wandering interval, it follows from Lemma~\ref{LemmaInt876f} that there exist sequences $a'_n<p<b'_n$ such that $b'_n\searrow p$ and $\co_f^+(b'_n)\cap(a'_n,b'_n)=\emptyset$.
 Thus, we set $(\alpha_n,\beta_n)=(a_n,b_n)\cap(a'_n,b'_n)$,  whenever $p\in\omega_f(p)$.

Notice that $(\alpha_n,\beta_n)$ is always a nice interval and $|(\alpha_n,\beta_n)|\searrow0$.

Let $n_0\ge1$ be such that $\delta:=\min\{|R|/|U|,|L|/|U|\}>0$, where $R,L$ are the connected components of $T\setminus(\alpha_{n_0},\beta_{n_0})$.
For each $n\ge n_0$ let $U_n=\{x\in T\,;\,\co_f^+(x)\cap(\alpha_n,\beta_n)\ne\emptyset\}$ and $\ell_n(x)=\min\{j\ge0$ $;$ $f^j(x)\in(\alpha_n,\beta_n)\}$, whenever $x\in U_n$. Let $F_n:U_n\to (\alpha_n,\beta_n)$ be the first entry map of $T$ in $(\alpha_n,\beta_n)$, i.e., $F_n(x)=f^{\ell_n(x)}(x)$. 

As $F_n$ is a first entrance map, if $V$ is a connected component of $U_n$, then $f^j(V)\cap f^k(V)=\emptyset$ $\forall0\le j<k\le \ell_n(V)$. In particular, $\sum_{j=0}^{\ell_n(V)}|f^j(V)|\le1$.

Let $T_n$ be the maximal open interval containing $J$ such that $f^{\ell_n(J)}|_{T_n}$ is a homeomorphism and that $f^{\ell_n(J)}(T_n)\subset T$.   As $T\cap\overline{\co_f^+(\cv_f)}=\emptyset$, we get that $f^{\ell_n(J)}(T_n)=T$. Let $J_n$ be the connected component of $U_n$ containing $J$. In particular, $\ell_n(J_n)=\ell_n(J)$. As $(\alpha_n,\beta_n)$ is a nice interval, $(F_n)(J_n)=(\alpha_n,\beta_n)$.
For each $n\ge n_0$, let $R_n$ and $L_n$ be the connected components of $f^{R_n(J_n)}(T_n)\setminus F_n(J_n)$ $=$ $T\setminus (\alpha_n,\beta_n)$.
So, $\min\{|R_n|/|F_n(J_n)|,|L_n|/|F_n(J_n)|\}\ge\delta>0$. As $\varepsilon:=\max\{|f^j(T_n)|\,;\,0\le j\le \ell_n(J_n)\}\le1$, it follows from Proposition~\ref{KoebeVvS} that there exists $\gamma>0$ such that 
\begin{equation}\label{EqBDoiuhg}
|DF_n(x)|\big/|DF_n(y)|\le \gamma,
\end{equation}
for every $x,y\in J_n$ and $n\ge n_0$.

Since $\lim_{n}|(\alpha_n,\beta_n)|=0$, it follows that $\lim_{n} \ell_n(J_n)=\lim_{n} \ell_n(J)=\infty$.
Therefore, as $J_{n}\supset J_{n+1}\supset J$ $\forall\,n\ge n_0$ and $f^k|_{J'}$ is a diffeomorphism for all $k\ge1$, where $J':=\interior(\bigcap_{n\ge n_0}J_n)$. As a consequence, it follows from Lemma~\ref{LemmaHomterval}, the homterval lemma, that $J'$ is a wandering interval. By the maximality of $J$, we have $J'=J$. In particular, $\lim_n\leb(J_n\setminus J)=0$. Thus, it follows from the bounded distortion (\ref{EqBDoiuhg}) that
\begin{equation}\label{Eqkjd992dbop}\lim_{n}\frac{|F_n(J)|}{|(\alpha_n,\beta_n)|}=\lim_n\frac{|F_n(J)|}{|F_n(J_n)|}=1.
\end{equation}

Now, let $\cu_n=\{x\in(\alpha_n,\beta_n)$ $;$ $\co_f^+(f(x))\cap(\alpha_n,\beta_n)\ne\emptyset\}$ and $\cf_n:\cu_n\to(\alpha_n,\beta_n)$ be the first return map to $(\alpha_n,\beta_n)$. Let $r_n(x):=\min\{j\ge1$ $;$ $f^j(x)\in(\alpha_n,\beta_n)\}$, $x\in\cu_n$, be the first return time to $(\alpha_n,\beta_n)$. That is, $\cf_n(x)=f^{r_n(x)}(x)$.
Let $\mathcal{J}_n$ be the connected component of $\cu_n$ containing $F_n(J)$ and let $\mathcal{T}_n$ be the maximal open interval containing $\mathcal{J}_n$ such that $f^{r_n(\mathcal{J}_n)}|_{\mathcal{T}_n}$ is monotone and $f^{r_n(\mathcal{J}_n)}(\mathcal{T}_n)\subset T$.
Again, as $(\alpha_n,\beta_n)$ is a nice interval and $(\alpha_n,\beta_n)\cap\RE(\co_f^+(\cv_f))=\emptyset$, for all $n\ge n_0$, we get that $\cf_n(\mathcal{T}_n)=T$ and $\cf_n(\mathcal{J}_n)=(\alpha_n,\beta_n)$, $\forall\,n\ge n_0$.

Because $\cf_n$ is a first return map, $f^j(\mathcal{J}_n)\cap f^k(\mathcal{J}_n)=\emptyset$ for all $0\le j<k\le r_n(\mathcal{J}_n)$ and so, $\sum_{j=0}^{r_n(\mathcal{J}_n)}|f^j(\mathcal{J}_n)|\le1$. As $|f^j(\mathcal{T}_n)|\le1$, $0\le j\le r_n(\mathcal{J}_n)$, it follows from Proposition~\ref{KoebeVvS} that
$$
|D\cf_n(x)|\big/|D\cf_n(y)|\le \gamma,
$$
for every $x,y\in\mathcal{J}_n$ and $n\ge n_0$.
This non-linearity distortion control combined with (\ref{Eqkjd992dbop}) implies that, if $n$ is big enough, $$f^{r_n(\mathcal{J}_n)+\ell_n(J)}(J)\cap f^{\ell_n(J)}(J)=\cf_n(F_n(J))\cap F_n(J)\ne\emptyset,$$
which is impossible as $J$ is a wandering interval.

Now, we consider the second case, that is, suppose that $p$ belongs to the boundary of a wandering interval. As we are assuming that  $p\in\overline{\co_f^+(I)\cap(0,p)}$, $p$ has to be the point at the left of the boundary of the wandering interval. Thus, let $W=(p,q)$ be the maximal wandering interval such that $p\in\partial W$. Notice that $p\in\omega_f(W)$.

\begin{claim}
For every $\varepsilon>0$, there are $a\in(p-\varepsilon,p)$ and $b\in(q,q+\varepsilon)$ such that $(a,b)$ is a nice interval.
\end{claim}
\begin{proof}[Proof of the Claim]
Let $(a,b)$ be the maximal open interval containing $J$ such that $\{p,q\}\cap\omega_f(x)\ne\emptyset$ for all $x\in(a,b)$.
Following the argument in the proof of Lemma~\ref{LemmaInt876f}, we can conclude that $(a,b)$ is a wandering interval.
By the maximality  of $W$, we get that $(a,b)=W$.
Therefore, for every $\varepsilon>0$ there exists $p-\varepsilon<\alpha<p$ and $q<\beta<q+\varepsilon$ such that $\{p,q\}\cap\overline{\co_f^+(\alpha)\cup\co_f^+(\beta)}=\emptyset$.
Thus, the connected component $(a,b)$  of $(p-\varepsilon,q+\varepsilon)\setminus\overline{\co_f^+(\alpha)\cup\co_f^+(\beta)}$ containing $(p,q)$ is a nice interval.
\end{proof}

For each $n\ge1$, let $(\alpha_n,\beta_n)$ be a nice interval contained in (p-1/n,q+1/n).
As before, let $n_0\ge1$ be such that $\delta:=\min\{|R|/|U|,|L|/|U|\}>0$, where $R,L$ are the connected components of $T\setminus(\alpha_{n_0},\beta_{n_0})$. Let $F_n:U_n\to(a_n,b_n)$ be the first entry map in $(\alpha_n,\beta_n)$, $U_n=\{x\in T$ $;$ $\co_f^+(x)\cap(\alpha_n,\beta_n)\ne\emptyset\}$, $\ell_n(x)$ be the first entry time, i.e., $F_n(x)=f^{\ell_n}(x)(x)$, and let $J_n$ be the connected component of $U_n$ containing $J$.
As before, see (\ref{EqBDoiuhg}), we get that
$$
|DF_n(x)|\big/|DF_n(y)|\le \gamma,\,\,\forall\,x,y\in J_n,\,\,\forall\,n\ge n_0.
$$
Because $\bigcap_{n}(\alpha_n,\beta_n)=(p,q)$, $\lim_n\ell_n(J_n)=\infty$ and, as before, we can conclude that $\lim_n\leb(J_n\setminus J)=0$. Again, from the non-linearity distortion control as above, it follows that $\lim_n\frac{|F_n(J)|}{|(\alpha_n,\beta_n)|}=1$. As a consequence,
\begin{equation}\label{EG0djndi}
f^{\ell_n}(J_n)(J)\cap W\ne\emptyset
\end{equation}
for every $n$ big enough and this leads to a contradiction. Indeed, if $f^{\ell_{n_1}}(J_n)(J)\cap(p,q)\ne\emptyset$, then it follows from the maximality of $W$ that $f^{\ell_{n_1}}(J_n)(J)\subset W$.
So, as $W$ is a wandering interval, we get that $f^{n}(J)\cap W\subset f^{n}(W)\cap W=\emptyset$ $\forall\,n>\ell_{n_1}$, which contradicts (\ref{EG0djndi}).
\end{proof}

\begin{Corollary}\label{LemNice8}
	If $p\in[0,1]\setminus\overline{\co_f^+(\cv_f)}$ is not an attracting periodic-like point, then for each $\delta>0$ there is a nice interval $J=(\alpha,\beta)$ such that $p-\delta<\alpha<p<\beta<p+\delta$.
\end{Corollary}
\begin{proof}
If $p\notin L_+(f)=\overline{\bigcup_{x}\omega_f(x)}$, then let $a_0\in(p-\delta,p)\setminus\co_f^-(p)$ and $b_0\in(p,p+\delta)\setminus\co_f^-(p)$.
Let $J_0$ be the connected component of $[0,1]\setminus\overline{\co_f^+(a_0)}$ containing $p$ and let $J_1$ be the connected component of $[0,1]\setminus\overline{\co_f^+(b_0)}$ containing $p$.
Setting $J=J_0\cap J_1$, we finish the proof for $p\notin L_+(f)$.

Now, suppose that $p\in L_+(f)$.
We may assume that $p\in\overline{\co_f^+(q)\cap[0,p)}$, the case $p\in\overline{\co_f^+(q)\cap(p,1]}$ being  analogous. 
As $p\in\overline{\co_f^+(q)\cap[0,p)}$, $(a,p)$ cannot be a wandering interval for $0<a<p$. 
As $p$ is not an attracting periodic point, it follows from Lemma~\ref{LemmaInt876f} that, given any $\delta>0$, $\exists\,a_0\in(p-\delta,p)$ such that  $p\notin\overline{\co_f^+(a_0)}$.
Let $(a_1,b_1)$ be the connected component of $[0,1]\setminus\overline{\co_f^+(a_0)}$ containing $p$.
As $(a_1,b_1)$ is a nice interval and as $p-\delta<a_0\le a_1<p<b_1$, then if $b_1<p+\delta$, the proof is finished.
So, let us assume that $b_1\ge p+\delta$.

If $(p,p+r)$ is a wandering interval for some $0<r\le\delta$, then it follows from Proposition~\ref{PropWanderingFree} that $p\notin\omega_f((p,b_1))$, where $b_2=\min\{p+r,p+\delta/2\}$. So, letting $a_2=\sup(\co_f^+((p,b_1))\cap[0,p))$, we have that $(a_2,b_2)$ is a nice interval. Thus, taking $J=(a_2,b_2)\cap(a_1,b_1)$, the proof is finished.
\end{proof}

\section{Induced Markov maps}

\begin{Lemma}\label{LemmaRetnicgdt5}
Let $f:[0,1]\setminus\cc_f\to[0,1]$ be a $C^2$ non-flat local diffeomorphism, where $\cc_f\subset[0,1]$ is a finite set. If $q\in L_+(f)\setminus\overline{\co_f^+(\cv_f)}$ is not an attracting periodic-like point, then given any $\varepsilon>0$ there is a nice interval $J=(\alpha,\beta)$ with $\alpha<q<\beta$ and such that
\begin{enumerate}
	\item  $|R|/|J|$ and $|L|/|J|\ge1$, where $R$ and $L$ are the connected components of $T\setminus\ J$ and $T$ is the connected component of $[0,1]\setminus\overline{\co_f^+(\cv_f)}$ containing $q$.
	\item $\big|\frac{F'(x)}{F'(y)}\big|\le1+\varepsilon^2$ for every $x,y\in I$ and every connected component $I$ of $J^*$, where $F:J^*\to J$ is the first return map to $J$ and $J^*=\{x\in J\,;\,\co_f^+(f(x))\cap J\ne\emptyset\}$.
\end{enumerate}
\end{Lemma}
\begin{proof}
For each $0<\delta<\dist(q,\partial T)/4$, let $J_{\delta}=(\alpha_\delta,\beta_\delta)\subset(q-\delta,q+\delta)$, with $\alpha_{\delta}<q<\beta_{\delta}$, being a nice interval given by Corollary~\ref{LemNice8}.
So, if $T$ is a connected component of $[0,1]\setminus\overline{\co_f^+(\cv_f)}$ containing $q$, we get $|R|/|J_{\delta}|\ge 1$ and $|L|/|J_{\delta}|\ge1$, where $R$ and $L$ are the connected components of $T\setminus\ J_{\delta}$.

Let $\ell_{\delta}:\cu_{\delta}\to\NN$ be the first entry time to $J_{\delta}$ and  $G(x)=f^{\ell_{\delta}(x)}(x)$ the first entry map to $J_{\delta}$, where $\cu_{\delta}=\{x\,;\,\co_f^+(x)\cap J_{\delta}\ne\emptyset\}$. Notice  that $\cu_{\delta}\ne\emptyset$, because $p\in L_+(f)$.

Let $\cp_G(\delta)$ be the collection of all connected components of $\cu_{\delta}$, and $|\cp_G(\delta)|=\max\{|I|\,;\,I\in\cp_G(\delta)\}$.  As $J_{\delta}$ is a nice interval and $\co_f^+(\cv)\cap J_{\delta}=\emptyset$, we have that $G_{\delta}|_I$ is a homeomorphism between $I$ and $J_{\delta}$, for every $I\in\cp_G(\delta)$.
\begin{claim}
$\lim_{\delta\to0}|\cp_G(\delta)|=0$.
\end{claim}
\begin{proof}[Proof of the Claim]
Note that if $0<\delta_0<\delta_1$, then each $I_0\in\cp_G(\delta_0)$ is contained in some $I_1\in\cp_G(\delta_1)$.  Thus,
if $|\cp_G(\delta)|\not\to0$ then there exists a sequence $\delta_n\searrow 0$ and $I_n\in\cp_G(\delta_n)$ such that $I_1\supset I_2\supset I_3\supset\cdots$ and $I:=\bigcap_n I_n$ is a nontrivial interval, i.e., $|I|>0$.
As $|J_{\delta_n}|\searrow0$, it follows that $\ell_{\delta_n}(I_n)\to\infty$. Furthermore, as $\ell_{\delta_n}(I_n)\to\infty$ and as $f^{\ell_{\delta_n}(I_n)}|_I$ is a diffeomorphism for all $n\ge1$, we get that $f^n|_I$ is a homeomorphism for every $n\in\NN$. That is, $I$ is a homterval.

As $q\in\omega_f(I)$, since $f^{\ell_{\delta_n}(I_n)}(I)\subset J_{\delta_n}\to q$ and as $q$ is not an attracting periodic point, it follows that $I\not\subset\BB_0(f)\cup\co_f^-(\per(f))$. Thus, from Lemma~\ref{LemmaHomterval}, we conclude that $I$ is a wandering interval, contradicting Proposition~\ref{PropWanderingFree}.
\end{proof}

Let $F_{\delta}:J_{\delta}^*\to J_{\delta}$ be the first return map to $J_{\delta}$ and $r_{\delta}:J_{\delta}^*\to\NN$ be the first return time, where $J_{\delta}^*=\{x\in J_{\delta}\,;\,\co_f^+(f(x))\cap J_{\delta}\ne\emptyset\}$. Let $\cp_F(\delta)$ be the collection of all connected components of $J_{\delta}^*$. Again, because $q\in L_+(f)$, $\cp_F(\delta)\ne\emptyset$.

Choose $0<\delta_1<\dist(q,\partial T)/4$ small enough so that $e^{O(|\cp_G(\delta_1)|)}<1+\varepsilon^4$, where $O$ is the function that appears in Proposition~\ref{KoebeVvS}. Let $\delta_2\in(0,\delta_1)$ be such that $\gamma:=\min\{|A|/|J_{\delta}|,|B|/|J_{\delta}|\}$ satisfies $\big(\frac{1+\gamma}{\gamma}\big)^2<1+\varepsilon^4$, where $A$ and $B$ are the connected components of $J_{\delta_1}\setminus J_{\delta_2}$.

Consider some $I\in\cp_F(\delta_2)$. Let $V\supset I$ be the maximal interval such that $f^{r_{\delta_2}(I)}(V)\subset J_{\delta_1}$ and that  $f^{r_{\delta_2}(I)}|_V$ is a homeomorphism of $V$ with $f^{r_{\delta_2}(I)}(V)$.
As $\co_f^+(\cv_f)\cap J_{\delta_1}=\emptyset$, it follows that $f^{r_{\delta_2}(I)}(V)=J_{\delta_1}$.
As $J_{\delta_1}$ is a nice interval and $G_{\delta_1}$ is the first entry map to $J_{\delta_1}$, we get that for each $i\in\{0,\cdots,r_{\delta_2}(I)\}$ there is a $W_i\in\cp_G(\delta_1)$ such that $f^i(V)\subset W_i$.
Thus, $\max\{|V|,\cdots,|f^{r_{\delta_2}(I)}(V)|\}\le|\cp_G(\delta_2)|$.
Moreover, as $F_{\delta_2}$ is a first return map, $f^i(I)\cap f^k(I)=\emptyset$, $\forall\,0\le i<k<r_{\delta_2}(I)$, and so, $\sum_{j=0}^{r_{\delta_2}(I)-1}|f^j(I)|\le 1$.
As a consequence, applying Proposition~\ref{KoebeVvS}, we get that $\big|\frac{DF_{\delta_2}(x)}{DF_{\delta_2}(y)}\big|\le(1+\varepsilon^4)(1+\varepsilon^4)<1+\varepsilon^2.$
\end{proof}

\color{black}

\begin{Lemma}\label{Lemmajhvtr68ijn}
Let $\varepsilon>0$ and let $g:[a,b]\to\RR$ be a orientation preserving diffeomorphism such that $g$ has a unique fixed point which it is either $a$ or $b$. Let $|g'(x)-1|<\varepsilon$ $\forall\,x$.
Let $G:\bigcup_{n\ge1}A_n\to J$ be the first entry map with respect to $g$ of $[a,b]$ in the interval $J:=(g(a),g(b))\setminus(a,b)$, where $A_n=g^{-n}(J)$.
If  $\exists\,K>0$ such that $G'(x)/G'(y)<K$ $\forall\,x,y\in A_n$, $\forall\,n\ge1$, then $G'(x)\ge \frac{1}{\varepsilon K}\frac{|J|}{|b-a|}$.
\end{Lemma}
\begin{proof}We may suppose that $g(a)=a$, the other case being analogous. In such a case, $J=(b,c)$ with $c=g(b)>b$.
Writing $a_n:=g^{-n}(c)$, we get $A_n=(a_{n+1},a_n)$ (see Figure~\ref{ladeira.png}). It follows from the mean value theorem that $1+\frac{|A_n|}{a_n-a}=\frac{a_n-a}{a_{n+1}-a}<1+\varepsilon$.
\begin{figure}
  \begin{center}\includegraphics[scale=.25]{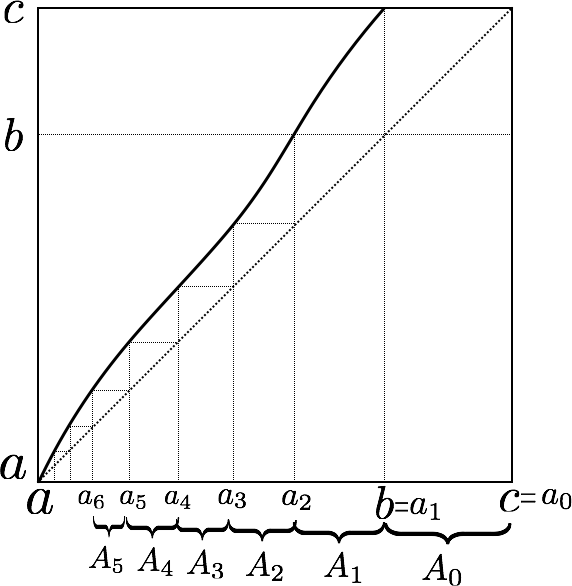}\\
 \caption{}\label{ladeira.png}
  \end{center}
\end{figure}
Thus, $|A_n|/|b-a|\le\frac{|A_n|}{a_n-a}<\varepsilon$, $\forall\,n\ge1$. As $g^n(A_n)=(b,c)$, we get $G'(x)=(g^n)'(x)\ge \frac{1}{K}\frac{|c-b|}{|A_n|}$ $>$ $\frac{1}{\varepsilon K}\frac{|c-b|}{|b-a|}$ $\forall\,x\in A_n$, $\forall\,n\ge1$.
\end{proof}

In Proposition~\ref{Propositionjhvtr68ijn} below let $f:[0,1]\setminus\cc_f\to[0,1]$ be a $C^2$ non-flat local diffeomorphism, where $\cc_f\subset[0,1]$ is a finite set.
Let  $T$ be a connected component of $[0,1]\setminus\overline{\co_f^+(\cv_f)}$ and $J\subset T$ be a nice interval such that $|R|/|J|,|L|/|J|\ge1$, where $R$ and $L$ are the connected components of $T\setminus\ J$.
 Let $F:J^*\to J$ be the first return map to $J$, where $J^*=\{x\in J\,;\,\co_f^+(f(x))\cap J\ne\emptyset\}$  and $\cu(J)=\big\{x\in[0,1]\setminus\big(\BB_0(f)\cup\co_f^-(\per(f))\big)\,;\,\omega_f(x)\cap J\ne\emptyset\big\}$.

\begin{Proposition}\label{Propositionjhvtr68ijn}
Let $K=5\exp({O(1)})>1$, where $O(\varepsilon)$ is given by Proposition~\ref{KoebeVvS}. 
Suppose that there exists $0<\varepsilon<(6K)^{-1}$ such that $\big|\frac{F'(x)}{F'(y)}\big|\le1+\varepsilon^2$ for every $x,y\in I$ and every connected component $I$ of $J^*$.
\begin{enumerate}
\item If $x\in\cu(J)$ does not belong to the pre-image of a non-hyperbolic periodic point then $\lim_{n\to\infty}|DF^n(x)|=\infty$.
\item $\omega_f(x)$ is a cycle of intervals containing $x$, for Lebesgue almost every $x\in\cu(J)$.
\end{enumerate}
\end{Proposition}
\begin{proof}
Firstly we will assume only that $|F'(x)/F'(y)|\le1+\varepsilon^2$ for every $x,y\in I$ and every connected component $I$ of $J^*$.
As $J\cap\overline{\co_f^+(\cv_f)}=\emptyset$, $F(I)=J$ for every connected component $I$ of $J^*$.
It follows that if $|F'(x)|>1+\varepsilon^2$, $\forall x\in J^*$, then $\lim |(F^n)'(x)|=\lim(1+\varepsilon^2)^n=\infty$,  for every $x\in\cu(J)=\bigcap_{\ell=0}^{\infty}(F^{\ell})(J)$.
Thus, we may assume that $|F'(p)|\le1+\varepsilon^2$ for some $p\in J^*$.

Let $I_p^0$ be the connected components of $J^*$ containing $p$ and $I_p=(F|_{I_p^0})^{-1}(I_p^0)$.
From the distortion control and as $\varepsilon<1/6$, $F'(x)$ and $(F^2)'(x)\le(1+\varepsilon^2)^3<1+\varepsilon$ for every $x\in I_p$.
In particular, we get that  $\frac{\leb(J\setminus I_p)}{\leb(J)}<\varepsilon$ and also that $F^2|_{I_p}$ preserves orientation. Let $a=\inf(\fix(F^2)\cap I_p)$ and $b=\sup(\fix(F^2)\cap I_p)$. Notice that $(a,b)\subset\BB_0(f)\cup\per(f)$, see Figure~\ref{TeoremaMane3bmaps.png}. 
Let $I_0=(a_0,b_0)$ and $I_1=(a_1,b_1)$ be the connected components of $J\setminus I_p$ and let $J_0=(b_0,a)$ and $J_1=(b,a_1)$ be the connected components of $I_p\setminus(a,b)$.

\begin{figure}
  \begin{center}\includegraphics[scale=.32]{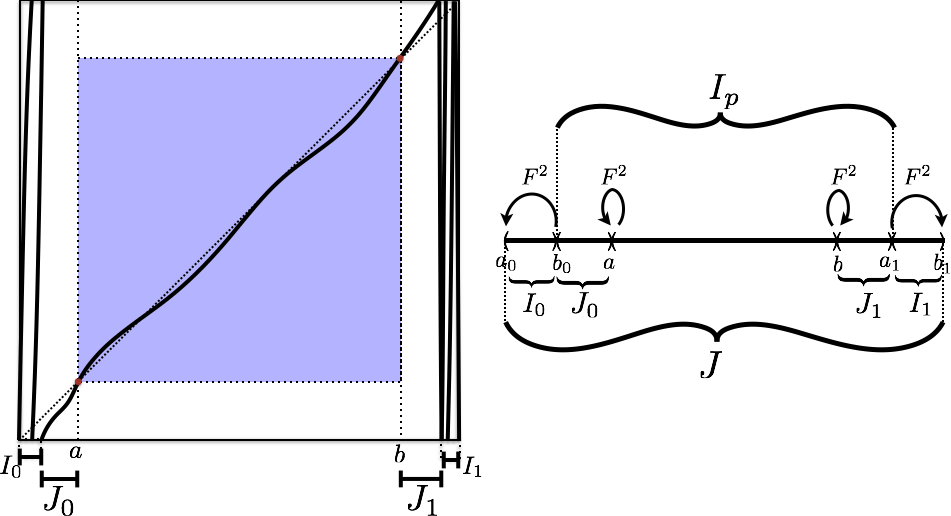}\\
 \caption{\small The picture at the right is the graph of $F^2$. The one at the left displays the essential elements discussed in the proof of Proposition~\ref{Propositionjhvtr68ijn}.}\label{TeoremaMane3bmaps.png}
  \end{center}
\end{figure}

Also let $g_0:[b_0,a]\to\RR$ be the first entry map of $[b_0,a]=\overline{J_0}$ to $[a_0,a]=\overline{I_0\cup J_0}$. Analogously, let $g_1:\overline{J_1}\to\RR$ be the first entry map of $\overline{J_1}$ to $\overline{I_1\cup J_1}$.
Notice that $g_j=F|_{\overline{I_j\cup J_j}}$ if $F|_{I_p^0}$ preserves orientation. Otherwise, we take $g_j=(F|_{\overline{I_j\cup J_j}})^2$. In particular, $|(g_j)'-1|<\varepsilon$.
Let $G_j:J_j^{*}\to I_j$ be the first entry map with respect to $f$ of $J_j$ into $I_j$, where $J_j^{*}$ $=$ $\{x\in J_j\,;\,\co_f^+(x)\cap I_j\ne\emptyset\}$. 
As $G_j$ is a first entry map, with respect to $f$, of $J_j$ to $I_j\subset J\subset T\subset[0,1]\setminus\overline{\co_f^+(\cv_f)}$, it follows from Proposition~\ref{KoebeVvS} that $|DG_j(x)|/|DG_j(y)|\le 4 \exp({O(1)})$, for every $x$ and $y$ in the same connected component of $J_j^*$.
Therefore, 
$|DG_j(x)|/|DG_j(y)|\le 4(1+\varepsilon^2)\exp({O(1)})<K$, for every $x$ and $y$ in the same connected component of $J_j^*$.
Applying Lemma~\ref{Lemmajhvtr68ijn}, we obtain that $G_j'(x)\ge\frac{1}{\varepsilon K}\frac{|I_j|}{|J_j|}$ for every $x\in J_j^*$.

Now consider $\cf_j:I_j^*\to I_j$ to be the first return map of $F$ to $I_j$, where $I_j^*$ $=$ $\{x\in I_j\,;\,\co_F^+(F(x))\cap I_j\ne\emptyset\}$.
As $F$ is the first return map to $J\supset I_0\cup I_1$ with respect to $f$, we can write $\cf_j=F^{R_j(x)}(x)=f^{r_j(x)}(x)$ with $r_j,R_j:I_j^*\to\NN$ and $1\le R_j(x)\le r_j(x)<\infty$.
Note that given $x\in I_j^*$, there are $n\ge0$ and $\alpha_0,\cdots,\alpha_n,\beta_0,\cdots,\beta_n\in\NN$ with $\alpha_j\ge1$ and $\beta_j=0$ or $1$ such that $$\cf_j(x)=
\begin{cases}
G_0^{\beta_0}\circ (G_1^{\beta_n}\circ H_1^{\alpha_n})\circ\cdots\circ(G_1^{\beta_1}\circ H_1^{\alpha_1})\circ H_0^{\alpha_0}(x) &\text{ if }j=0\\
G_1^{\beta_0}\circ (G_0^{\beta_n}\circ H_0^{\alpha_n})\circ\cdots\circ(G_0^{\beta_1}\circ H_0^{\alpha_1})\circ H_1^{\alpha_0}(x) &\text{ if }j=1
\end{cases},
$$
where $H_j:=F|_{I_J}$.

As $|DH_j(x)|\ge \frac{1}{1+\varepsilon}\frac{|J_j|+|I_j|}{|I_j|}>\frac{1}{1+\varepsilon}\frac{|J_j|}{|I_j|}$, we get 
$$|DG_j^{\beta_{\ell}}(x)|\,|DH_j^{\alpha_{\ell}}(y)|\ge\frac{1}{\varepsilon(1+\varepsilon)K}>3,$$
for every $x\in J_j$, $y\in I_j$, $j\in\{0,1\}$ and $\ell\in\{0,...,n\}$. As a consequence,
\begin{equation}\label{EqManeoiuy7}|D\cf_j(x)|>3,\;\forall\,x\in I_j^*.
\end{equation}
In particular,
$\lim_{n\to\infty}|D(\cf_j(x))^n|=\infty$ $\forall\,x\in \bigcap_{\ell=0}^{\infty}(\cf_j)^{-\ell}(I_j)$.
Moreover,
as $$\cu(J)\setminus(\co_f^{-}(a)\cup\co_f^{-}(b))=\bigg(\bigcap_{\ell=0}^{\infty}(\cf_0)^{-\ell}(I_-)\bigg)\cup\bigg(\bigcap_{\ell=0}^{\infty}(\cf_1)^{-\ell}(I_1)\bigg),$$ either $\lim|(F^n)'(x)|=\lim|(\cf_i^n)'(x)|=\infty$ for some $i\in\{0,1\}$, or $x\in\co_f^-(u)$ with $u\in\{a,b\}$ being a non-hyperbolic periodic point.
This concludes the proof of the first item of the proposition. 


Now, we shall prove item (2) of the lemma. For that,
suppose that $\leb(\cu(J))>0$. If $|F'|>1+\varepsilon^2$ then set $\cf:=F$. Otherwise, $\leb(\cu(J)\cap I_0)>0$ or $\leb(\cu(J)\cap I_1)>0$, and in this case, consider any $\ell\in\{0,1\}$ such that $\leb(\cu(J)\cap I_{\ell})>0$ and set $\cf:=\cf_{\ell}$. So, let
$$
\cR=\begin{cases}
r &\text{ if }\cf=F\\
R_{\ell} &\text{ if }\cf=\cf_{\ell}
\end{cases},\;\;
\ci=\begin{cases}
J &\text{ if }\cf=F\\
I_{\ell} &\text{ if }\cf=\cf_{\ell}
\end{cases}
\;\;\text{ and }\;\;
\ci^*=\begin{cases}
J^* &\text{ if }\cf=F\\
I_{\ell}^* &\text{ if }\cf=\cf_{\ell}
\end{cases}.
$$

Observe that
\begin{equation}\label{Eqpoihg6y3nmdo}
|D\cf(x)|>1+\varepsilon^2\,\,\forall\,x\in\ci^*.
\end{equation}

As $\partial I_{\ell}\cap J^*=\emptyset$ and $J$ is a nice interval, it follows that $\ci$ is also a nice interval. Thus, $\cf$ is a {\em full induced Markov map}, i.e., $\cf(U)=\ci$ for every connected component $U$ of $\ci^*$

Set $\cp_n$ as the collection of all connected component of $\cf^{-n}(\ci^*)$. If $x\in \cf^{-n}(\ci^*)$, let $\cp_n(x)$ be the element of $\cp_n$ containing $x$. So, we always have $\cf^{n+1}(\cp_n(x))=\ci$.

Because $F$ is the first return map of $f$ to $J$ and that either $\cf$ is $F$ or it is the return map of $F$ to $I_{\ell}$, it follows that $\cf$ itself is the first return map of $f$ to $\ci\subset J\subset T\subset[0,1]\setminus\overline{\co_f^+(\cv_f)}$. Thus, from Proposition~\ref{KoebeVvS}, we obtain that
\begin{equation}\label{Eqijs552vcc}
|(f^j)'(x)/(f^j)'(y)|\le K_0
\end{equation}
for every $x,y\in I$, $I\in\cp_0$ and $1\le j\le\cR(I)$, where $K_0=\big(\frac{1+\delta'}{\delta'}\big)^2\exp(O(1))$.
Furthermore,
\begin{equation}\label{Eqkjdg13653}
|\cf'(x)/\cf'(y)|\le e^{\gamma_0\sum_{j=0}^{\cR(I)}|f^j(x)-f^j(y)|},
\end{equation}
for every $x,y\in I$ and $I\in\cp_0$, where $\gamma_0=O(1)+2/|J|$.

As a consequence of the bounded distortion (\ref{Eqijs552vcc}), we have the following Claim.
\begin{claim}
$\sum_{j=1}^{\cR(I)}\leb(f^j(V))\le \frac{K_0}{|\ci|}\leb(\cf(V))$ for every Borel set $V\subset I$ and $I\in\cp_0$.
\end{claim}
Indeed, it follows from (\ref{Eqijs552vcc}) that $$\frac{\leb{(f^j(V))}}{\leb(f^j(I))}\le K_0\frac{\leb(\cf(V))}{\leb(\cf(I))}=K_0\frac{\leb(\cf(V))}{|\ci|},$$
for every  $1\le j\le\cR(I)$. Thus,
$$\sum_{j=1}^{\cR(I)}\leb(f^j(V))=\sum_{j=1}^{\cR(I)}\frac{\leb(f^j(V))}{\leb(f^j(I))}\leb(f^j(I))\le$$
$$\le\sum_{j=1}^{\cR(I)}\frac{K_0}{|\ci|}\leb(\cf(V))\leb(f^j(I))=\frac{K_0}{|\ci|}\leb(\cf(V))\sum_{j=1}^{\cR(I)}\leb(f^j(I))\le \frac{K_0}{|\ci|}\leb(\cf(V)).$$

From the Claim above and (\ref{Eqkjdg13653}), it follows that 
$$
|\cf'(x)/\cf'(y)|\le e^{\gamma|\cf(x)-\cf(y)|},
$$
for every $x,y\in I$ and $I\in \cp_0$, where $\gamma=K_0\gamma_0/|\ci|$.
Therefore, using the expansion (\ref{Eqpoihg6y3nmdo}), we get
\begin{equation}\label{Eqoiuh8778}
\bigg|\frac{(\cf^n)'(x)}{(\cf^n)'(y)}\bigg|\le e^{\gamma\sum_{j=0}^{n-1}|\cf^j(x)-\cf^j(y)|}\le e^{\gamma\sum_{j=0}^{n-1}|\cf^n(x)-\cf^n(y)|/(1+\varepsilon^2)^{j}}\le\Gamma,
\end{equation}
for every $x,y\in\cp_n(q)$ and $q\in\bigcap_{j\ge0}\cf^{-j}(\ci)$, where $\Gamma=e^{\gamma(1+1/\varepsilon^2)}$.
Finally, applying Lemma~\ref{pagina1}, we get that $\omega_{\cf}(x)=\overline{\ci}$ for almost all $x\in\ci$. In particular, $\omega_f(x)$ is a cycle of intervals containing $x$, for almost every $x\in\cu(J)\cap\ci$. If $\ci=J$ the proof finished. The proof is also concluded if $\ci\ne J$, i.e., when $|Df(p)|\le1+\varepsilon^2$ for some $p\in J$, because in this case $\cu(J)\cap(a,b)=\emptyset$, that is, $\cu(J)=\cu(J)\cap\overline{I_0}\cup\cu(J)\cap\overline{I_1}$.\end{proof}

\section{Proofs of Theorem~\ref{MainTheoPreMane} and Corollary~\ref{ThMane} (Mañé's Theorem)}

In this section we provide a new proof of Mañé's theorem mentioned before. To do so, we show the existence of induced Markov maps and use them to prove the hyperbolicity of the points that avoid the critical set. Notice that typically one uses some kind of hyperbolicity to build up Markov partitions or induced Markov maps, in quite the opposite way to what we are doing here.

\begin{Remark}\label{RemarkExt001}
Given any piecewise $C^2$ local diffeomorphism $f:[0,1]\setminus\cc_f\to[0,1]$, where $C_f\subset(0,1)$ is finite, we can obtain an extension $\widetilde{f}:[-1,2]\setminus\cc_f\to[-1,2]$ of $f$ that is also a local diffeomorphism with the same exceptional set $\cc_f$ and satisfying the conditions below (see Figure~\ref{Extensao.pdf}).  
\begin{enumerate}
	\item $\widetilde{f}(\{-1,2\})\subset\{-1,2\}$.
	\item If $x\in(-1,0)\cup(1,2)$ then either $\omega_{\widetilde{f}}(x)\subset\{-1,2\}$ or  $\omega_{\widetilde{f}}(x)\subset[0,1]$.
\end{enumerate}

\begin{figure}
  \begin{center}\includegraphics[scale=.18]{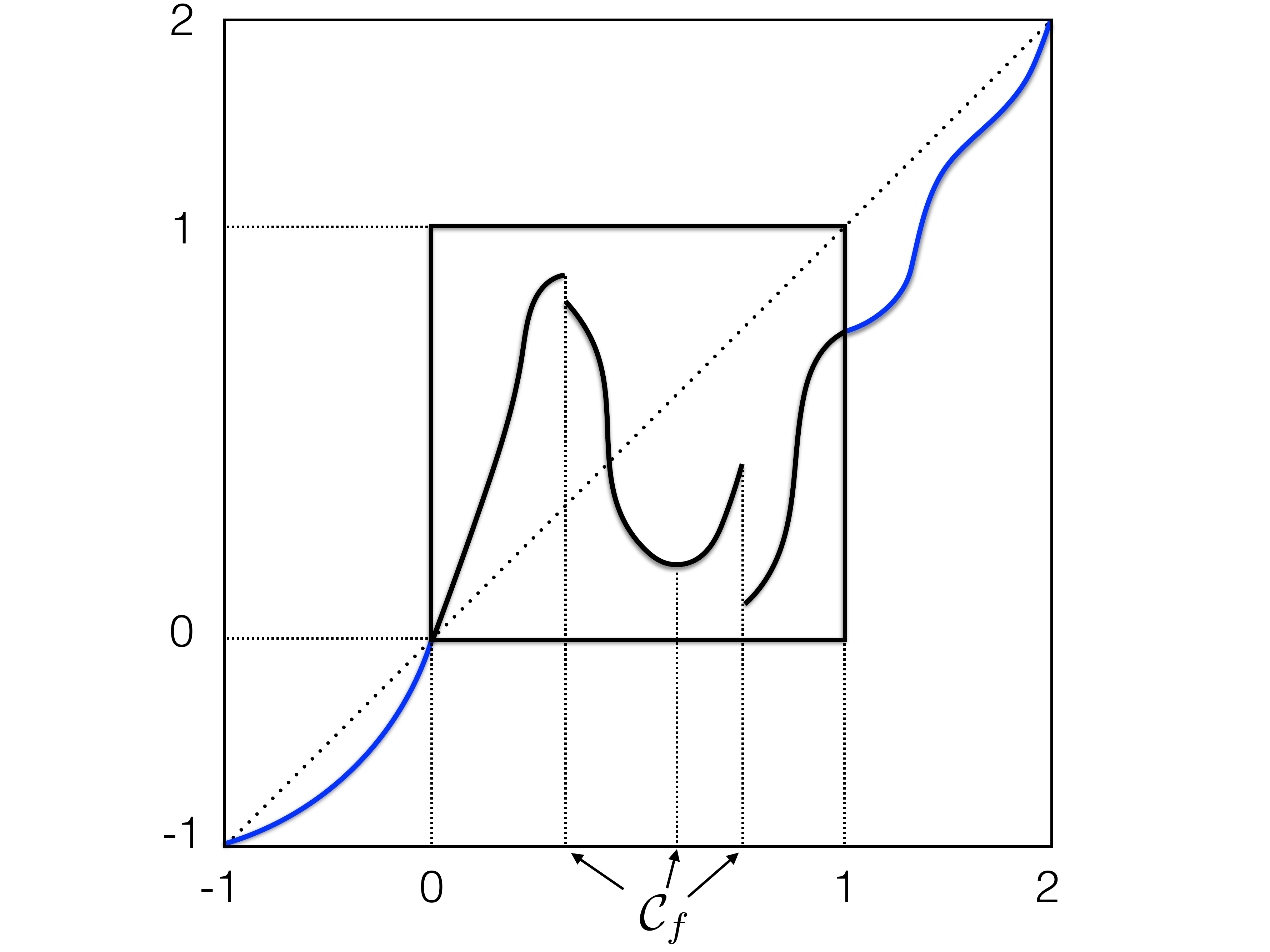}\\
 \caption{}\label{Extensao.pdf}
  \end{center}
\end{figure}

In particular, $\widetilde{f}$ has the same non-periodic attractors and at most two more attracting periodic orbits which are contained in $\{-1,2\}$.

\end{Remark}

\begin{Lemma}\label{Lemmaiuywe4jj}
Let $p$ be a periodic-like point and $q\in[0,1]\setminus(\BB_0(f)\cup\co_f^-(\per(f)))$.
\begin{enumerate}
\item If $p$ is a left side periodic-like point then $$p\in\overline{\co_f^+(q)\cap[0,p)} \iff p\in\overline{(\omega_f(q)\setminus\per(f))\cap[0,p)}.$$
\item If $p$ is a right side periodic-like point then $$p\in\overline{\co_f^+(q)\cap(p,1]}\iff p\in\overline{(\omega_f(q)\setminus\per(f))\cap(p,1]}.$$
\end{enumerate}
\end{Lemma}
\begin{proof}
Let us prove the first item, the poof of the second one is analogous.
As ``$\Leftarrow$'' is immediate, we may assume that  $p$ is a left side periodic-like point and $p\in\overline{\co_f^+(q)\cap[0,p)}$. Let $\ell\ge1$ be such that $f^\ell(p_-)=p_-$, that is, $f^\ell(x)\nearrow x$ when $x\nearrow p$.
Thus, there exists $0<a<p$ such that $f^\ell|_{[a,p)}$ is a preserving orientation diffeomorphism and $\lim_{x\to p}f^\ell|_{[a,p)}(x)=p$.

Given $\delta>0$, we need to show that $(\omega_f(q)\setminus\per(f))\cap(p-\delta,p)\ne\emptyset$. 
Notice that if $f^{\ell}(p')=p'$ for some $p'\in[a,p)$, then $f^{\ell}|_{[p',p)}$ can be extended to a diffeomorphism $F$ of $[p',p]$ into itself with $F(p')=p'$ and $F(p)=p$.
Thus, $[p',p]\subset\BB_0(F)\cup\co_F^-(\per(F))\subset\BB_0(f)\cup\co_f^-(\per(f))$.
As $\co_f^+(q)\cap(p',p)\ne\emptyset$, we get a contradiction. 

So, as $f^{\ell}|_{[a,p)}$ preserves orientation, either $f^{\ell}(x)<x$ $\forall\,x\in[a,p)$ or $f^{\ell}(x)>x$ $\forall\,x\in[p-r,p)$. As $f^{\ell}(x)>x$ $\forall\,x\in[a,p)$ implies that $p_-$ is a periodic-like attractor, we get that $f^{\ell}(x)<x$ $\forall\,x\in[a,p)$.

Given any $\delta\in(0,p-a)$, write $p_\delta=p-\delta$ and $p_\delta'=(f^{\ell}|_{[a,q)})^{-1}(p_\delta)$. 
Consider a sequence $n_j\to\infty$ such that $(p_\delta,p)\ni q_j:=f^{n_j}(q)\nearrow p$.
Note that, for each $j\ge1$ there exists a unique integer $k_j\ge0$ such that $f^{k_j\ell}(q_j)\in[p_\delta,p_\delta']$ and $f^{i\,\ell}(q_j)\in[p_\delta,p)$ for every $0\le i<k_j$.
As a consequence, $\#\big(\co_f^+(q)\cap[a,a']\big)=\infty$.
Thus, there is some $q_\delta\in[p_\delta,p_\delta']\cap\omega_f(q)$. If $q_\delta\notin\per(f)$ the proof is complete.

So, we may suppose that $q_{\delta}\in\per(f)$ and let $s$ be the period of $q_{\delta}$ with respect to $f^{\ell}$. 
For each $i>s$, let $q_{\delta,i}=(f^{\ell}|_{[a,p)})^{-i}(q_\delta)$. As $q_{\delta,i}=\lim_{j\to\infty}f^{(k_j-i)\ell}(q_j)$, we get that $q_{\delta,i}\in(p-\delta,p)\cap\omega_f(q)$ $\forall\,i>s$. On the other hand, as $f^{i\ell}(q_{\delta,i})=q_\delta$ and $f^{n\ell}(q_{\delta,i})\ne q_{\delta,i}$ for every $0\le n< i$ with $i>s$, we get that $q_{\delta,i}$ is a pre-periodic (but not periodic) point of $\omega_f(p)$, which concludes the proof.

\end{proof}

\begin{Corollary}\label{Lemmaprepre987}
Let $\cc_f\subset(0,1)$ be a finite set and $f:[0,1]\setminus\cc_f\to[0,1]$ a local homeomorphism. If $p\in[0,1]\setminus\BB_0(f)\cup\co_f^-(\per(f))$ then $\omega_f(p)\not\subset\per(f)$.
\end{Corollary}

\begin{proof}[Proof of Theorem~\ref{MainTheoPreMane}]
	Let $p\in[0,1]\setminus\big(\BB_0(f)\cup\co_f^-(\per(f))\big)$ such that $\overline{\co_f^+(p)}\cap\cc_f=\emptyset$.
	Extending $f$ if  necessary, as in Remark~\ref{RemarkExt001}, we may assume that $f( \{0,1\}) \subset \{0,1\}$ and that $\overline{\co_f^+(p)}\cap\{0,1\}=\emptyset$.
	Let $U$ be an open neighborhood of $\cc_f$ such that $\overline{\co_f^+(p)}\cap \overline{U}=\emptyset$.
	Let $g:[0,1]\setminus\cc_g\to[0,1]$ be a $C^2$ non-flat local diffeomorphism such that $g|_{[0,1]\setminus U}\equiv f|_{[0,1]\setminus U}$, $\cc_g=\cc_f$ and $\cv_{g}=\{g(c_{\pm})\,;\,c\in\cc_g\}\subset\{0,1\}$ (see Figure~\ref{Extensao3.png}).
	
	As $\overline{\co_g^+(\cv_g)}\subset\{0,1\}$, we have $\overline{\co_g^+(p)}\cap\overline{\co_g^+(\cv_g)}=\overline{\co_f^+(p)}\cap\{0,1\}=\emptyset$.
	By Corollary~\ref{Lemmaprepre987}, consider $q\in\omega_g(p)\setminus\per(g)$.
	Therefore, it follows from Proposition~\ref{Propositionjhvtr68ijn} that $$\sup_{n}|Df^n(p)|=\sup_{n}|Dg^n(p)|=\lim_n|DG^n(p)|=\infty,$$
	where $G$ is the first return time to a small nice interval $J=(\alpha,\beta)\ni q$ given by Lemma~\ref{LemmaRetnicgdt5}.
	This proves the first statement of Theorem~\ref{MainTheoPreMane}. 
	
	To prove the second statement, let $\Lambda$ be the set of all $x\in[0,1]\setminus\big(\BB_0(f)\cup\co_f^-(\per(f))\big)$ such that $\omega_f(x)\cap\cc_f=\emptyset$.
	Suppose that $\leb(\Lambda)>0$.
	As $f^*\leb\ll\leb$, $\leb(\Lambda)=\leb(\Lambda\setminus\co_f^-(\cv_f))$.
	So, $\exists\ell\ge1$ such that $\leb(\Lambda_\ell)>0$, where $$\Lambda_{\ell}=\{x\in\Lambda\setminus\co_f^-(\cc_f)\,;\,\overline{\co_f^+(x)}\cap B_{1/\ell}(\cc_f)=\emptyset\}.$$

Given any $\varepsilon>0$ and $p\in\Lambda_{\ell}$, it follows from Corollary~\ref{Lemmaprepre987},  Lemma~\ref{LemmaRetnicgdt5} and Proposition~\ref{Propositionjhvtr68ijn} that there exist a point $q_p\in\omega_f(p)\setminus\per(f)$ and a nice interval $J_p$ containing $q_p$ 
such that $\omega_f(x)$ is a cycle of intervals for almost every $x\in\cu(J_p)$, where $\cu(J_p):=\big\{x\in[0,1]\setminus\big(\BB_0(f)\cup\co_f^-(\per(f))\big)\,;\,\omega_f(x)\cap J_p\ne\emptyset\big\}$.

Consider now any sequence $p_n\in\Lambda_{\ell}$ such that $$\bigcup_{p\in\Lambda_{\ell}}J_p=\bigcup_{n\ge1}J_{p_n}$$ and let $W_n=\{x\in\Lambda_{\ell}\,;\,\omega_f(x)\cap J_{p_n}\ne\emptyset\}$. As $\Lambda_{\ell}=\bigcup_{n\ge1}W_n$, let $m\ge1$ be such that $\leb(W_m)>0$.
As $\cu(J_{p_m})\supset W_m$, it follows that there is a positive set of points $x$ of $\Lambda$ such that $\omega_f(x)$ is a cycle of intervals. This is a contradiction to the  well known fact that every cycle of intervals contains a point of $\cc_f$. Indeed, suppose that $I$ is an interval of a cycle of interval and that $f^J(I)\cap\cc_f=\emptyset$ $\forall\,n\ge1$. Thus, $I$ is a homterval that is not a wandering one. So, as a consequence of the Homterval Lemma, $I\subset\BB_0(f)\cup\co_f^-(\per(f))$, which is a contradiction. 
\begin{figure}
  \begin{center}\includegraphics[scale=.25]{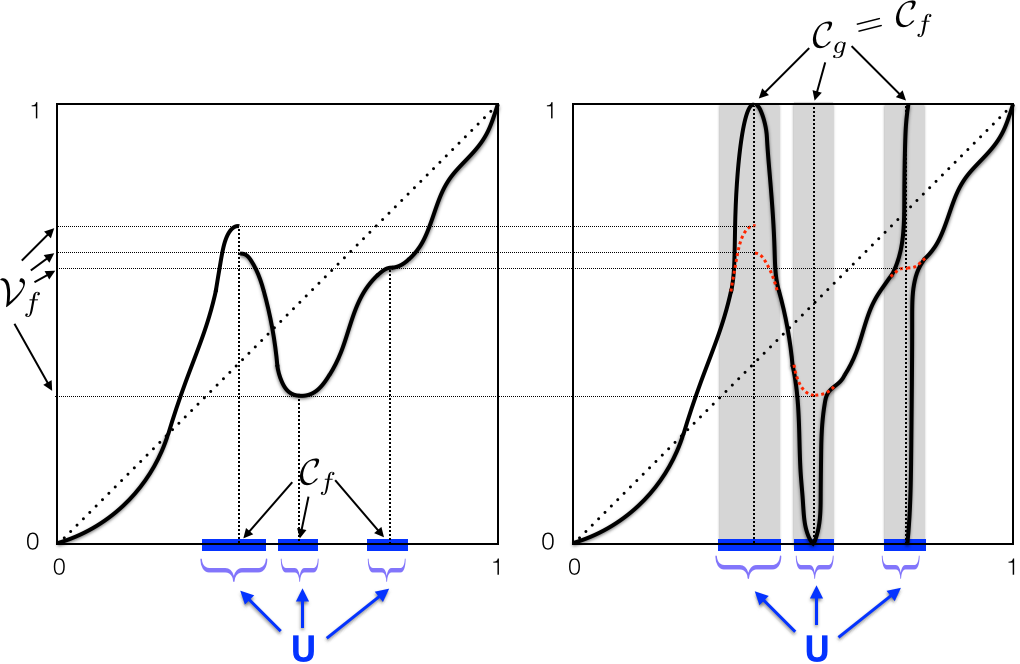}\\
 \caption{}\label{Extensao3.png}
  \end{center}
\end{figure}
\end{proof}

\begin{proof}[Proof of Corollary~\ref{ThMane}]

Let $\Lambda=\bigcup_{n\ge0}f^n([0,1]\setminus U)$. By hypothesis, if $x\in\Lambda\cap\co_f^-(\per(f))$ then $\sup_n|Df^n(x)|=\infty$.
On the other hand, if $x\in\Lambda\setminus\per(f)$, it follows from Theorem~\ref{MainTheoPreMane} that $\sup_n|Df^n(x)|=\infty$.
By compactness, there is $\ell\ge1$ such that $|Df^{\ell}(x)|\ge2$ for every $x\in\Lambda$.
This means that $\Lambda$ is uniformly expanding set, as stated in the Corollary.

Let $K=[0,1]\setminus U$. We claim that, for each $\varepsilon>0$ such that $B_{\varepsilon}(\Lambda)\subset K$, $\exists\,n_0$ with $\min\{j\ge1\,;\,f^j(x)\in U\}\le n_0$ for all $x\in B_{\varepsilon}(\Lambda)$, where $B_{\varepsilon}(\Lambda)=\bigcup_{p\in\Lambda}B_{\varepsilon}(p)$.
Otherwise, there is a sequence of points $x_n\in K\setminus B_{\varepsilon}(\Lambda)$ and a sequence of integers $0\le j_n\nearrow\infty$ such that $f^{i}(x_n)\in K$ for every $0\le i\le j_n$.
Thus, taking a subsequence, we may assume that $x_n$ converges to a point $p\in K\setminus B_{\varepsilon}(\Lambda)$.
But this implies that $f^i(p)\in K$, $\forall\,i\ge0$, which is not possible since $p\notin\Lambda=\bigcap_{j\ge0}f^j(K)$.

Let $C_0=\min\{|Df(x)|\,;\,x\in K\}>0$. As $Df$ is continuous on $K$, and as $\Lambda$ is uniformly expanding and invariant, there are $C_1>0$, $\lambda>1$ and $\varepsilon>0$ so that $B_{\varepsilon}(\Lambda)\subset K$ and $|Df^n(x)|\ge C_1\lambda^n$ whenever $f^j(x)\in B_{\varepsilon}(\Lambda)$ $\forall0\le j<n$.
Therefore, for every $x\in\bigcap_{i=0}^{n-1}f^{-i}(K)$, we get 
$|Df^n(x)|\ge C_0^{n_0}|Df^{n-n_0}(x)|\ge C\lambda^n$,
where $C=C_0^{n_0}C_1/\lambda^{n_0}$.

\end{proof}

\section{Finiteness of non-periodic attractors: proofs of Theorems~\ref{mainTheoremMTheoFofA} and \ref{mainTheoremClassification}}

Given a map $f$ of the interval $[0,1]$, let $\BB_1(f)$ be the set of all points $x$ such that $\omega_f(x)$ is a cycle of intervals. 

\begin{Lemma}\label{tudoounadadershchw} 
If $f:[0,1]\setminus\cc_f\to[0,1]$ is a $C^2$ non-flat local diffeomorphism, where $\cc_f\subset[0,1]$ is a finite set, then
$$\omega_f(x)\subset\overline{\co_f^+(\cv_f)},$$
for almost every $x\in[0,1]\setminus\big(\BB_0(f)\cup\BB_1(f)\cup\co_f^-(\per(f)\big))$.
\end{Lemma}
\begin{proof}
Suppose that there is a connected component $T$ of $[0,1]\setminus\overline{\co_f^+(\cv_f)}$ such that $\Lambda=\{x\in[0,1]\setminus(\BB_0(f)\cup\co_f^-(\per(f)))\,;\,\omega_f(x)\cap T\ne\emptyset\}$ has positive Lebesgue measure. As $[0,1]\setminus\overline{\co_f^+(\cv_f)}$ has only a countable number of connected components, we have only to show that $\leb(\Lambda\setminus\BB_1(f))=0$.

From now on, the remaining part of the proof is similar to the proof of the second statement of Theorem~\ref{MainTheoPreMane}. 
Indeed, given any $\varepsilon>0$ and $p\in\Lambda$, it follows from Corollary~\ref{Lemmaprepre987},  Lemma~\ref{LemmaRetnicgdt5} and Proposition~\ref{Propositionjhvtr68ijn} that there exist a point $q_p\in\omega_f(p)\setminus\per(f)$ and a nice interval $J_p$ containing $q_p$ 
such that $\omega_f(x)$ is a cycle of intervals for almost every $x\in\cu(J_p)$, where $\cu(J_p):=\big\{x\in[0,1]\setminus\big(\BB_0(f)\cup\co_f^-(\per(f))\big)\,;\,\omega_f(x)\cap J_p\ne\emptyset\big\}$.
As $\RR$ is a Lindelöf space, write $$\bigcup_{p\in\Lambda}J_p=\bigcup_{n\ge1}J_{p_n},$$ where $p_n\in\Lambda_{\ell}$. As $\Lambda=\bigcup_{n\ge1}W_n$, with $W_n=\{x\in\Lambda\,;\,\omega_f(x)\cap J_{p_n}\ne\emptyset\}$, we can conclude that almost every point of $\Lambda$ belongs to $\BB_1(f)$.

\end{proof}

\begin{Theorem}\label{TheoOMEGAtudo}
If $f:[0,1]\setminus\cc_f\to[0,1]$ is a $C^2$ non-flat local diffeomorphism, where $\cc_f\subset[0,1]$ is a finite set, then
$$
\omega_f(x)=
\bigcup_{\footnotesize{
\begin{array}{c}
c_{\pm}\in\omega_f(x)\\
c\in\cc_f\end{array}
}}
\overline{\co_f^+(c_{\pm})},
$$
for almost every $x\in[0,1]\setminus\big(\BB_0(f)\cup\BB_1(f)\cup\co_f^-(\per(f))\big)$.
\end{Theorem}
\begin{proof}
By Remark~\ref{RemarkExt001}, we may assume that $f^{-1}(\{0,1\})=\{0,1\}$ and that $\omega_f(x)\cap\{0,1\}=\emptyset$ for every $x\in(0,1)$.

For any $q\in\{c_{\pm}\,;\,c\in\cc_f\}$ and $\cu\subset\{c_{\pm}\,;\,c\in\cc_f\}$, let 
$$\XX(q)=\{x\in [0,1]\setminus\big(\BB_0(f)\cup\BB_1(f)\big)\,;\,q\notin\omega_f(x)\},$$
$$\AA(\cu)=\{x\in[0,1]\setminus(\BB_0(f)\cup\BB_1(f))\,;\,\cu\subset\omega_f(x)\}$$
and $$U(\cu)=\{x\in[0,1]\setminus\big(\BB_0(f)\cup\BB_1(f)\big)\,;\,\omega_f(x)\subset\overline{\co_f^+(\cu)}\}.$$

\begin{claim}
If $q\in\cu\subset\{c_{\pm}\,;\,c\in\cc_f\}$ then
$\omega_f(x)\subset\overline{\co_f^+(\cu\setminus\{q\})}$, for almost every $x\in U(\cu)\cap\XX(q)$.	
\end{claim}
\begin{proof}[Proof of the Claim]
If $q\in\omega_f(p)$ for some $p\in\cu\setminus\{q\}$, then $\overline{\co_f^+(q)}\subset\overline{\co_f^+(p)}$ and so, $\overline{\co_f^+(\cu)}\subset\overline{\co_f^+(\cu\setminus\{q\})}$, which proves the Claim. Thus, we may assume that $q\notin\omega_f(p)$ and $\overline{\co_f^+(q)}\not\subset\overline{\co_f^+(p)}$ $\forall\,p\in\cu\setminus\{q\}$.
Let us consider the case when  $q=c_-$ for some $c\in\cc_f$. The other case, when $q=c_+$, is similar. 

Now, consider $n_0\ge1$ big enough so that $(c-2/n_0,c)\cap\cc_f=\emptyset$ and $\co_f^+(\cu\setminus\{c_-\})=\emptyset$.
Let $\XX_n=\{x\in U(\cu)\,;\,\co_f^+(x)\cap(c-1/n,c)\ne\emptyset\}$.
For each $n\ge n_0$, let  
$g_n:[0,1]\setminus\cc_{g_n}\to[0,1]$ be a $C^2$ non-flat local diffeomorphism such that $g_n|_{[0,1]\setminus(c-1/n,\,c)}\equiv f|_{[0,1]\setminus(c-1/n,\,c)}$, $\cc_{g_n}=\cc_f$ and $g_{n}(c_-)\in\{0,1\}$ (see Figure~\ref{Extensao4.png}).

As $\omega_f(x)\subset(0,1)$ for every $x\in(0,1)$ and $g_n|_{\XX_n}\equiv f|_{\XX_n}$, we conclude that for every $n\ge n_0$ and  every $x\in\XX_n$ we have that 
$$\omega_f(x)=\omega_{g_n}(x)\subset\big([0,1]\setminus\{c_-,0,1\}\big)\cap\overline{\co_{g_n}^+(\cu)}=$$ 
$$=\big([0,1]\setminus\{c_-,0,1\}\big)\cap\bigg(\overline{\co_{g_n}^+(\cu\setminus\{c_-\})}\cup\overline{\co_{g_{n}}^+(c_-)}\bigg)=$$
$$=\big([0,1]\setminus\{c_-,0,1\}\big)\cap\bigg(\overline{\co_{f}^+(\cu\setminus\{c_-\})}\cup\{c_-,0,1\}\bigg)=\overline{\co_{f}^+(\cu\setminus\{c_-\})}.$$
As $\XX(c_-)=\bigcup_{n\ge n_0}\XX_n$, the Claim is proved. 
\begin{figure}
  \begin{center}\includegraphics[scale=.28]{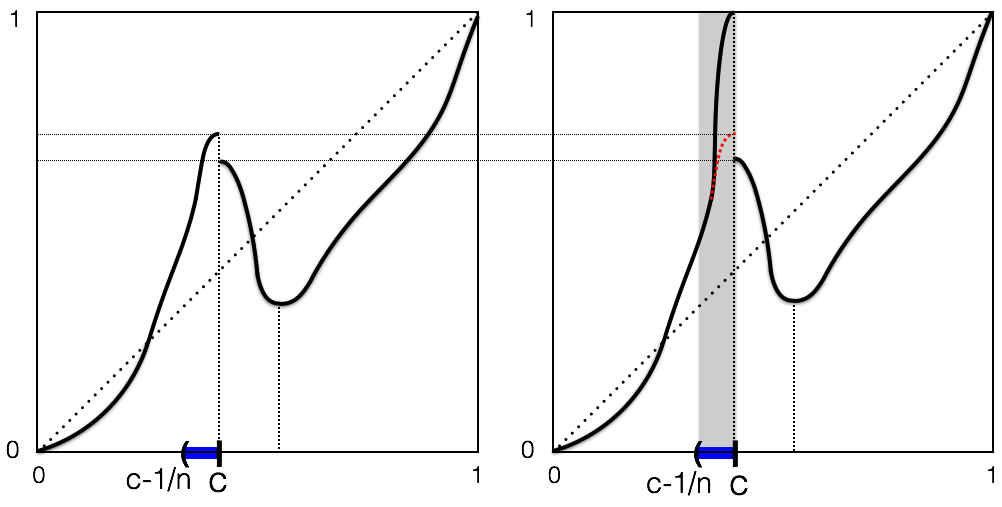}\\
 \caption{}\label{Extensao4.png}
  \end{center}
\end{figure}
\end{proof}

Given $\cu\subset\{c_{\pm}\,;\,c\in\cc_f\}$, it follows from Lemma~\ref{tudoounadadershchw} that $$\overline{\co_f^+(\cu)}\subset\omega_f(x)\subset\overline{\co_f^+(\cv_f)}\subset\overline{\co_f^+(\cc_f)},$$
for almost every point $x\in\AA(\cu)$. Applying recursively the Claim above, we get that 
$$\overline{\co_f^+(\cu)}\subset\omega_f(x)\subset\overline{\co_f^+(\cu)},$$
for almost every point $x\in\AA(\cu)$, concluding the proof.
\end{proof}

\begin{proof}[Proof of Theorem~\ref{mainTheoremMTheoFofA}]
	As we can have at most $2^{2\#\cc_f}-1$ non-empty subsets of $\{c_{\pm}\,;\,c\in\cc_f\}$, Theorem~\ref{mainTheoremMTheoFofA} is an immediate consequence of Theorem~\ref{TheoOMEGAtudo} and of the fact that a map with a finite number of exceptional points has only a finite numbers of cycles of intervals. \end{proof}


\begin{proof}[Proof of Theorem~\ref{mainTheoremClassification}]
	
Consider $\cu\subset\{c_{\pm}\,;\,c\in\cc_f\}$ such that $\leb(U)>0$, where $U=\{x\in[0,1]\setminus(\BB_0(f)\cup\BB_1(f))\,;\,\omega_f(x)=\overline{\co_f^+(\cu)}\}$.
Extending $f$ if  necessary, as in Remark~\ref{RemarkExt001}, we may assume that $f( \{0,1\}) \subset \{0,1\}$ and that $\overline{\co_f^+(x)}\cap\{0,1\}=\emptyset$ $\forall\,x\in U$.
Given $\alpha$ and $\beta\in\cu$, we say that $\alpha\prec\beta$ if $\alpha\in\omega_f(\beta)$ and $\beta\notin\omega_f(\alpha)$. A point $\alpha\in\cu$ is called {\em $\cu$-maximal} if $\nexists\beta\in\cu\setminus\{\alpha\}$ such that $\alpha\prec\beta$.
\begin{Claim}\label{Claim987ygvbnj}
	If $\alpha\in\cu$ is $\cu$-maximal then is not a periodic-like point.
\end{Claim}
\begin{proof}[Proof of the Claim]Suppose that $\alpha\in\cu$ is a periodic-like point.
Suppose that $\alpha=c_-$ for some $c\in\cc_f$, the proof for $\alpha=c_+$ is similar.
Given $x\in U$, it follows from Lemma~\ref{Lemmaiuywe4jj} that $c\in\overline{\omega_f(x)\cap[0,c)}$. As $\omega_f(x)=\overline{\co_f^+(\cu)}$ and $\cu$ is finite, there is $\beta\in\cu$ such that
\begin{equation}\label{Eqjdi2jnj}
c\in\overline{\overline{\co_f^+(\beta)}\cap[0,c)}.	
\end{equation}
As $\overline{\co_f^+(\alpha)}=\co_f^+(\alpha)$, it follows from (\ref{Eqjdi2jnj}) that $\alpha=c_-\in\omega_f(\beta)$ and also that $\omega_f(\beta)\ne\omega_f(\alpha)$. In particular, $\alpha\in\omega_f(\beta)$ but $\beta\notin\omega_f(\alpha)$. That is, $\alpha$ is not $\cu$-maximal.\end{proof}

\begin{Claim}\label{Claim987ygj}
	If $\alpha\in\cu$ is $\cu$-maximal then $\alpha$ is recurrent.
\end{Claim}
\begin{proof}[Proof of the Claim]
	Suppose by contradiction that there exists a $\cu$-maximal $\alpha\in\cu$ such that $\alpha\notin\omega_f(\alpha)$.
		Note that, in this case, $\alpha\notin\overline{\co_f^+(f(\cu))}=\bigcup_{\beta\in\cu}\overline{\co_f^+(\beta)}$.

		Given $\varepsilon>0$ and $c\in\cc_f$, let $I_{\varepsilon}(c_-)=(c-\varepsilon,c)$ and $I_{\varepsilon}(c_+)=(c,c+\varepsilon)$. Let $B_n=\bigcup_{\beta\in(\cc_f)_\pm\setminus\cu}I_{1/n}(\beta)$, for any $n\ge1$, $U_n=\{x\in U\,;\,\co_f^+(x)\cap B_n=\emptyset\}$. As $U=\bigcup_nU_n$, choose $\ell\ge1$ so that $\leb(U_\ell)>0$.

Let $g:[0,1]\setminus\cc_f\to[0,1]$ be any $C^2$ non-flat local diffeomorphism such that $g(x)=f(x)$ for every $x\in([0,1]\setminus\cc_f)\setminus B_\ell$ and that $g(\beta)\in\{0,1\}$ for every $\beta\in(\cc_f)_\pm\setminus\cu$. 

As in the Claim before, let us suppose that $\alpha=c_-$ for some $c\in\cc_f$, the proof for $\alpha=c_+$ is similar. Notice that $c\notin\overline{\co_g^+(g((\cc_f)_\pm))}=\bigcup_{a\in\cc_f}\overline{\co_g^+(c_\pm)}$, as $\overline{\co_g^+(g(\cu))}=\overline{\co_f^+(f(\cu))}$ and that $\overline{\co_g^+(g((\cc_f)_\pm\setminus\cu))}\subset\{0,1\}$.

Thus, we can apply Lemma~\ref{LemmaRetnicgdt5} and Proposition~\ref{KoebeVvS} to conclude that $\omega_g(x)$ is a cycle of intervals for almost every $x\in U_\ell$. This is a contradiction, as $\omega_f(x)=\omega_g(x)$ for every $x\in U_\ell$ and $U_\ell\cap\BB_1(f)=\emptyset$.

\end{proof}
	
As $\cu$ is a finite set, given any $\alpha\in\cu$ then either $\alpha$ is $\cu$-maximal or there is a $\cu$-maximal $\alpha'\in\cu$ such that $\alpha\prec\alpha'$.
Thus, there is a subset $\{\alpha_1,\cdots,\alpha_n\}\subset\cu$ of $\cu$-maximal points such that $\bigcup_{j=1}^n\overline{\co_f^+(\alpha_j)}=\bigcup_{\alpha\in\cu}\overline{\co_f^+(\alpha)}$.
As each $\alpha_j$ is recurrent (Claim~\ref{Claim987ygj}), we get that $\overline{\co_f^+(\alpha_j)}=\omega_f(f(\alpha_j))\ni\alpha_j$ $\forall\,1\le j\le n$.
Furthermore, it follows from Claim~\ref{Claim987ygj} that $v_j:=f(\alpha_j)\notin\co_f^-(\cc_f)$ is recurrent and not periodic.
Thus, as $f|_{\co_f^+(v_j)}$ is continuous, $v_j\notin\per(f)$ and $v_j\in\omega_f(v_j)$, we get that $\omega_f(v_j)$ is a perfect set.
Of course that $\interior(\omega_f(v_j))=\emptyset$, as $\omega_f(x)\supset\omega_f(v_j)$ $\forall\,x\in U$ and $U\cap\BB_1(f)=\emptyset$.
So, $\omega_f(v_j)$ is a Cantor set for every $1\le j\le n$. As a consequence, also $\omega_f(v_1)\cup\cdots\cup\omega_f(v_n)$ is a Cantor set, which concludes the proof.

\end{proof}

\color{black}

\end{document}